\documentclass[a4paper, twoside,11pt]{article}
\usepackage{fancyhdr}
\usepackage{fnpos}
\usepackage[english]{babel}        
\usepackage[T1]{fontenc}          
\usepackage{graphicx}             
\usepackage{makeidx}
\usepackage{fancybox}
\usepackage{framed}
\usepackage{fancyhdr}
\usepackage{pstricks,pst-plot,pstricks-add}
\usepackage[margin=1in]{geometry}
\usepackage{graphicx}
\usepackage{titlesec}
\usepackage{amsmath}
\usepackage{amsfonts}   
\usepackage{amssymb}    
\usepackage{amsthm}
\usepackage{dsfont}
\usepackage{mathtools}
\usepackage{verbatim}   
\usepackage{color}
\usepackage{listings}
\usepackage{graphicx}
\usepackage{amsmath}
\usepackage{amsmath,amssymb,color,inputenc,euscript,graphicx,psfrag}

\newtheorem{thm}{Theorem}[section]
\newtheorem{cor}[thm]{Corollary}
\newtheorem{lem}[thm]{Lemma}
\newtheorem{prop}[thm]{Proposition}
\newtheorem{defn}[thm]{Definition}
\newtheorem{rem}[thm]{Remark}

\numberwithin{equation}{section}

\DeclareMathOperator*{\supess}{sup\,ess}

\usepackage{amsmath, amsthm, amscd, amsfonts, amssymb, graphicx, color,   mathrsfs}
\usepackage[english]{babel}
\usepackage[bookmarksnumbered, colorlinks, plainpages]{hyperref}
\hypersetup{colorlinks=true,linkcolor=blue, anchorcolor=blue, citecolor=blue, urlcolor=red, filecolor=magenta, pdftoolbar=true}

 {\markboth{\sl A. Saoudi }{\sl  Continuous linear canonical Dunkl wavelet transform}
	
	\author {Ahmed Saoudi $^{\orcidlink{0000-0003-4422-2054}}$ and Imen Kallel$^{\orcidlink{0000-0001-6581-3182}}$}
	
	\title{Continous linear canonical Dunkl wavelet transform: properties and applications}
	
	\date{}

	\usepackage{orcidlink}

\begin{document}
		\maketitle
		\begin{center}
		$^a$Department of Mathematics, College of Science, Northern Border University, Arar, Saudi Arabia.\\
		Corresponding author: ahmed.saoudi@ipeim.rnu.tn\\
		E-mail: imenkallel16@gmail.com
		\end{center}
			\begin{abstract}
					The aim of this paper is to establish  and study the linear canonical Dunkl wavelet transform. We begin by introducing the generalized translation operator and generalized convolution product for  the linear canonical Dunkl transform and we establish their basic properties. Next, we introduce the new proposed wavelet transform and we investigate its fundamentals properties. 
					  In the end,  we derive  some uncertainty inequalities for the desired wavelet transform as applications. \\ \\
						\textbf{ Keywords}. linear canonical Dunkl transform; linear canonical Dunkl translation; linear canonical Dunkl convolution, linear canonical Dunkl wavelet transform; uncertainty principles. 
					
					\textbf{Mathematics Subject Classification}. Primary 43AXX; Secondary 42AXX.
				\end{abstract}
		
		\section{Introduction}
		
		The linear canonical wavelet transform (LCWT) is a mathematical tool used for signal analysis and processing. It is a linear and invertible transformation that decomposes a signal into a set of wavelet coefficients, providing a time-frequency representation of the signal.
		
	   Wavelets are functions that are used to analyze and represent signals in both the time and frequency domains. They are characterized by their ability to capture localized features of a signal at different scales. Unlike traditional Fourier analysis, which uses fixed-frequency sinusoids, wavelets have variable shapes and sizes, making them suitable for analyzing signals with transient or localized features.
		
		The LCWT is derived from the canonical wavelet transform (CWT), which is a technique for decomposing a signal into wavelet coefficients at different scales. The CWT is defined by convolving the signal with scaled and translated versions of a mother wavelet function. However, the CWT is not computationally efficient, as it involves an infinite number of convolutions. The LCWT addresses this limitation by discretizing the CWT. It achieves this by sampling the mother wavelet function at specific scales and positions, resulting in a finite set of wavelet coefficients. These coefficients represent the energy distribution of the signal in the time-frequency plane.
		
	
			
		The LCWT has various applications in signal processing, such as image compression, denoising, feature extraction, and pattern recognition \cite{annaby2022regularized}. It provides a powerful tool for analyzing signals with both localized and frequency-varying characteristics, making it well-suited for many real-world applications \cite{shah2021linear}.

		Recently, the linear canonical wavelet transform was studied in various setting such as localization operators and wavelet multipliers \cite{catanua2023localization}, in quaternion domains \cite{shah2021linear}, in  composition on generalized function spaces \cite{prasad2020composition}, in reproducing formula in the  Hankel setting \cite{prasad2018canonical}, in free metaplectic domains \cite{shah2022free}, in special affine domains \cite{lonea2023special, shah2023special} and linear canonical Fourier-Bessel domains \cite{mohamed2024linear}. Our goal in this paper is to establish a new generalized wavelet transform in linear canonical domains and investigate its fundamentals properties and we derive some uncertainty principles for this new wavelet transform. This new generalized wavelet transform  generalizes a large class of a family of wavelet transforms such as detailed in below figure.
				\vspace{-0.5cm}
		\begin{figure}[h]
			\centering
			\includegraphics[width=1\textwidth]{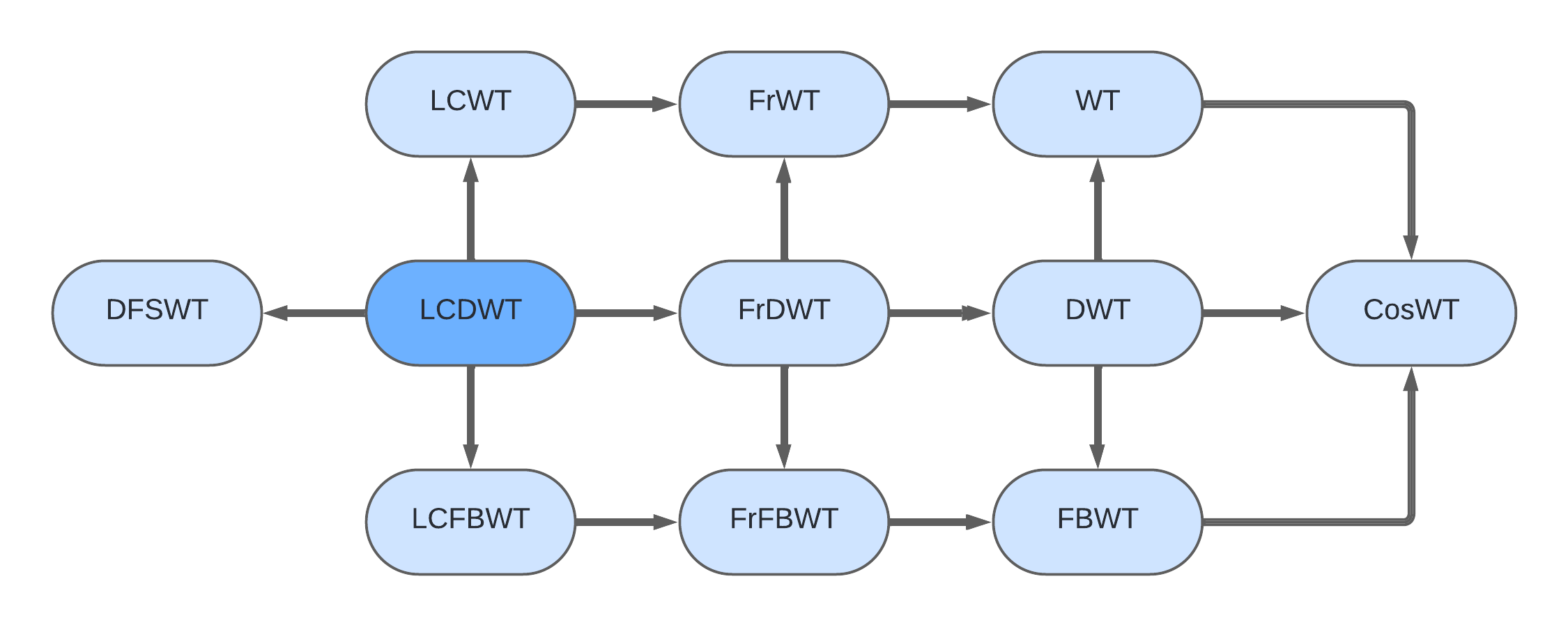}
			\caption{LCDWT generate a wide class of wavelet transforms.}
		\end{figure}
		\vspace{-0.5cm}
			We mention below the abbreviations used throughout above figure.\\~\\
		\textit{List of Abbreviations}\\
		LCDWT - linear canonical Dunkl wavelet transform\\
		DFSWT - Dunkl  Fresnel  wavelet transform\\
		LCWT - linear canonical wavelet transform\\
		LCFBWT - linear canonical Fourier--Bessel wavelet transform\\
		FrDWT - fractional Dunkl wavelet transform\\
		FrWT - fractional  wavelet transform\\
		FrFBWT - fractional Fourier--Bessel wavelet transform\\
		DWT -  Dunkl wavelet transform\\
		WT -  wavelet transform\\
		FBWT -  Fourier--Bessel wavelet transform\\
		CosWT - Fourier--cosine wavelet transform\\

		   The outline of this paper is as follows. In Sec.  \ref{sec2}, we give a brief  overview  for the linear canonical Dunkl transform. In Sec. \ref{sec3}, we  introduce the linear canonical Dunkl translation and convolution and we establish their basic properties. Section \ref{sec4} is devoted to define and study  the new proposed wavelet transform and  investigate its fundamentals properties. In the last section, we drive  some uncertainty inequalities for the desired wavelet transform as applications.

		\section{Overview on linear canonical Dunkl transform}\label{sec2}
		
%
		
		We are interested in this paper to the  one-dimensional canonical Dunkl transform $\mathcal{D}_{\mu}^{M}$, which generalizes a large class of a family of integral transforms,  see \cite{Ghazouani-unified17,HardyLCDT}.
		\begin{equation}\label{CDT}
			\mathcal{D}_{\mu}^{M} f(x)=\left\{\begin{array}{ll}
				\displaystyle\frac{1}{\Gamma(\mu+1)(2 i b)^{\mu+1}} \int_{\mathbb{R}}\mathcal{R}_{\mu}^{M}(x, y) f(y)|y|^{2 \mu+1} d y, & b \neq 0 \\\\
				\displaystyle\frac{e^{\frac{i}{2} \frac{c}{a} x^{2}}}{|a|^{\mu+1}} f(x / a), & b=0
			\end{array}\right.
		\end{equation}
		where $\mathcal{R}_{\mu}^{M}(x, y)$ is the  kernel defined by (see \cite[eq.3.30]{Ghazouani-unified17})
		\begin{equation}\label{kernelEmnu}
			\mathcal{R}_{\mu}^{M}(x, y)=e^{\frac{i}{2}\left(\frac{d}{b} x^{2}+\frac{a}{b} y^{2}\right)}\mathcal{R}_{\mu}(-i x / b, y)
		\end{equation}
		with $\mathcal{R}_{\mu}(x, y)$ is the Dunkl kernel of type $A_{2}$ given by (see \cite{Rosler2003})
	\begin{equation}\label{dkernel}
		\mathcal{R}_{\mu}(i x, y)=j_{\mu}(x y)+\frac{i x y}{2(\mu+1)} j_{\mu+1}(x y)
	\end{equation}
		and $j_{\mu}$ is the normalized spherical Bessel function
\begin{equation}\label{seriesb}
			j_{\mu}(x):=2^{\nu} \Gamma(\mu+1) \frac{J_{\mu}(x)}{x^{\mu}}=\Gamma(\mu+1) \sum_{n=0}^{\infty} \frac{(-1)^{n}(x / 2)^{2 n}}{n ! \Gamma(n+\mu+1)}.
\end{equation}
		Note that $J_{\mu}$ is the classical Bessel function (see \cite{Watson1944}).

Along this paper, we denote by  $L^{p}_{\mu}(\mathbb{R})$ the Lebesgue space of measurable functions on $\mathbb{R}$ such that
			$$
			\begin{array}{l}
				\|f\|_{\mu,p}=\left(\displaystyle\int_{\mathbb{R}}|f(y)|^{p} |y|^{2\mu+1} \mathrm{~d} y\right)^{\frac{1}{p}}<\infty, \quad \text { if } 1 \leq p<\infty \\
				\|f\|_{\mu,\infty}=\displaystyle\supess_{y\in\mathbb{R}}|f(y)|<\infty, \quad \text { if } p=\infty,
			\end{array}
			$$
			provided with the topology defined by the norm $\|\cdot\|_{\mu ,p}$
		and   $L^{2}_{\mu}(\mathbb{R})$ the Hilbert space equipped with the inner product $\langle., .\rangle_{\mu}$ given by
			$$
			\langle f, g\rangle_{\mu}=\int_{\mathbb{R}} f(y) \overline{g(y)} |y|^{2\mu+1} \mathrm{~d}y.
			$$
			
%

	 \begin{prop}\label{propLCDT}\cite{Ghazouani-unified17} 	Let $M$  be a matrix belongs to  $ S L(2, \mathbb{R})$.
	\begin{enumerate}
		\item  Suppose that $b \neq 0$. Then for all $f \in L_k^1\left(\mathbb{R}\right), \mathcal{D}_{\mu}^{M}f$ belongs to $\mathcal{C}_0\left(\mathbb{R}\right)$ and verifies
		$$
		\left\|\mathcal{D}_{\mu}^{M} f\right\|_{\mu,\infty} \leq \frac{1}{\Gamma(\mu+1)(2|b|)^{\mu+1}}\|f\|_{\mu,1}
		$$
		\item For all $f$ in $L^{1}_{\mu}(\mathbb{R})$ with $\mathcal{D}_{\mu}^{M} f \in L_{\mu}^{1}\left(\mathbb{R}\right)$,
	\begin{equation}\label{invlcdt}
		\left(\mathcal{D}_{\mu}^{M^{-1}} \circ \mathcal{D}_{\mu}^{M}\right) f=f \quad \text {, a.e., and }\quad \left(\mathcal{D}_{\mu}^{M}\circ \mathcal{D}_{\mu}^{M^{-1}}\right) f=f \text {, a.e. }
	\end{equation}
		\item The linear canonical Dunkl transform $\mathcal{D}_{\mu}^{M}$ is a one-to-one and onto mapping from $\mathcal{S}\left(\mathbb{R}\right)$ into itself. Moreover,
		 we have for all
		$f \in \mathcal{S}\left(\mathbb{R}\right)$
		$$
		\left(\mathcal{D}_{\mu}^{M}\right)^{-1} f=\mathcal{D}_{\mu}^{M^{-1}} f.
		$$
		\item Plancherel theorem: Let $f$ and $g$ be in $L^{1}_{\mu}(\mathbb{R})$ and $M \in S L(2, \mathbb{R})$ such that $b \neq 0$. Then
		$$
		\int_{\mathbb{R}} \mathcal{D}_{\mu}^{M}f(x) \overline{g(x)} |x|^{2\mu+1} d x=\int_{\mathbb{R}} f(x) \overline{\mathcal{D}_{\mu}^{M^{-1}} g(x)}|x|^{2\mu+1} d x.
		$$
		\item  Let $f \in \mathcal{S}\left(\mathbb{R}\right)$ and $M \in S L(2, \mathbb{R})$. Then
		$$
		\left\|\mathcal{D}_{\mu}^{M}f\right\|_{\mu,2}=\|f\|_{\mu,2}.
		$$
		\item Let $M \in S L(2, \mathbb{R})$.
		 If $f \in L^{1}_{\mu}(\mathbb{R}) \cap L^{2}_{\mu}(\mathbb{R})$, then $\mathcal{D}_{\mu}^{M}f \in L^{2}_{\mu}(\mathbb{R})$ and $\left\|\mathcal{D}_{\mu}^{M}f\right\|_{\mu,2}=\|f\|_{\mu,2}$.
		\item Let $M \in SL(2, \mathbb{R})$. There exists a unique unitary operator on $SL(2, \mathbb{R})$ which  coincides with $\mathcal{D}_{\mu}^{M}$ on $L^{1}_{\mu}(\mathbb{R}) \cap L^{2}_{\mu}(\mathbb{R})$, and we  denoted  this extension by $f \longrightarrow \mathcal{D}_{\mu}^{M}f$.
	\end{enumerate}
\end{prop}

\section{Linear canonical Dunkl translation and convolution}\label{sec3}
 In order to develop the new proposed  wavelet transform  we shall introduce the linear canonical Dunkl translation and convolution and establish some of their basic properties which will be very fruitful for the development of desired wavelet transform.
\begin{defn}
Let $M=\left(  a,b,c,d\right)  \in SL(2,\mathbb{R})$ such that $b\neq0$ and $f\in\mathscr C(\mathbb{R})$. We define  the linear canonical Dunkl translation operator $\mathcal{T}_{x}^{\mu, M}\left(f\right)(y)$ for all $x,y\in\mathbb{R}$  by
	\begin{equation}\label{FJTransl}
		\mathcal{T}_{x}^{\mu, M}\left(f\right)(y)=\int_{\mathbb{R}} e^{-i\frac{d}{b}z^{2}}f(z) f(z) \mathcal{W}^M_{\mu}(x, y, z) |z|^{2\nu +1} dz,
	\end{equation}
	where
	\begin{equation}
		\mathcal{W}_{\mu}^{M}(x,y,z)=e^{\frac{i}{2}\frac{d}{b}%
			(x^{2}+y^{2}+z^{2})}\mathcal{W}_{\mu}(x,y,z) \label{4.2}%
	\end{equation}
	with  $\mathcal{W}_{\mu}(x, y, z)$ is a nonnegative, symmetric function in three variables, compactly supported in $z$ and for more details see \cite[Theorem 2.2]{Dunkl-transform-TBM}.
\end{defn}
It is easy to check by simple calculation the following properties for the linear canonical Dunkl translation operator.
\begin{lem}\label{propertiestransl} Let $M=\left(  a,b,c,d\right)  \in SL(2,\mathbb{R})$ such that $b\neq0$ and $f\in\mathscr C(\mathbb{R})$. The linear canonical Dunkl translation operator $\mathcal{T}_{x}^{\mu, M}$, verifies the following properties:
	\begin{enumerate}
		\item Identity: 	$\mathcal{T}_{0}^{\mu, M}\left(f\right)(y)=f(y).$
		\item Symmetry:  	$\mathcal{T}_{x}^{\mu, M}\left(f\right)(y)= \mathcal{T}_{y}^{\mu, M}\left(f\right)(x).$
		\item  Linearity: $\mathcal{T}_{x}^{\mu, M}\left(cf +g\right)(y)= c\mathcal{T}_{x}^{\mu, M}\left(f\right)(y)+ \mathcal{T}_{x}^{\mu, M}\left(g \right)(y).$
		\item  Product formula: $\mathcal{T}_{x}^{\mu, M}\left(\mathcal{R}_{\mu}^{M}(\cdot, y)\right)(z)=e^{-\frac{i}{2}\frac{a}{b}y^{2}}\mathcal{R}_{\mu}^{M}(x, y)\mathcal{R}_{\mu}^{M}(z, y)$.
		\item Continuity: for all  $f\in L^{p}_{\mu}(\mathbb{R})$, $1\leq p\leq \infty$ and $x\in\mathbb{R},$ the function $\mathcal{T}_{x}^{\mu, M}f$ is defined
		almost everywhere on $\mathbb{R}$, belongs to $ L^{p}_{\mu}(\mathbb{R})$ and we have  $\|\mathcal{T}_{x}^{\mu, M}\left(f\right)\|_{p, \mu}\leq \|f\|_{\mu ,p}$.
	\end{enumerate}
\end{lem}

Now, using  the above proposed translation operator we shall define the linear canonical Dunkl convolution product and we give for it some properties.
\begin{defn} Let $M \in SL(2,\mathbb{R})$ such that $b\neq0$ and let $f$ and $g$ be measurable functions on $\mathbb{R}$.  We define  the linear canonical Dunkl convolution product of $f$ and $g$ by
	\begin{eqnarray}
		f\star g(x)=\int_{\mathbb{R}}\left[\mathcal{T}_{0}^{\mu, M}f\right](y) \ \left[e^{-i\frac{d}{b}y^{2}}g(y)\right]|y|^{2\mu +1}dy
		\label{convo}
	\end{eqnarray}
	for all $x$ such that the integral exists.
\end{defn}
\begin{lem} Assume that all integrals in question exist, then we have  the elementary properties of convolutions product. 
	\begin{enumerate}
		\item   $f \star g =g \star f .$ 
		\item 	$\mathcal{T}_{x}^{\mu, M}\left(f \star g \right)=\left[\mathcal{T}_{x}^{\mu, M}f \right]\star g =f \star \left[\mathcal{T}_{x}^{\mu, M}g \right].$
	\end{enumerate}
\end{lem}

\begin{prop} (\textbf{Young's type Inequality}) Let $M \in SL(2,\mathbb{R})$ such that $b\neq0$, $1 \leq p, q, r \leq \infty$ satisfy $\frac{1}{p}+\frac{1}{q}-1=\frac{1}{r}$,  $f  \in L^{p}_{\mu}(\mathbb{R})$ and $g  \in L^{q}_{\mu}(\mathbb{R})$. Then the linear canonical Dunkl convolution  product $f \star g $ belongs to $\in L^{r}_{\mu}(\mathbb{R})$, and we have
	$$
	\left\|f \star g \right\|_{\mu , r}^{2} \leq C\|f \|_{\mu ,p}^{2}\|g \|_{\mu ,q}^{2}.
	$$
	\label{young1}
\end{prop}	
\begin{proof} Let $M \in SL(2,\mathbb{R})$ such that $b\neq0$, $1 \leq p, q, r \leq \infty$ satisfy $\frac{1}{p}+\frac{1}{q}-1=\frac{1}{r}$,  $f  \in L^{p}_{\mu}(\mathbb{R})$ and $g  \in L^{q}_{\mu}(\mathbb{R})$. We apply the  inequality of H\"{o}lder to following product, we get
	$$
	\left|\mathcal{T}_{x}^{\mu, M}f(y)e^{-i\frac{d}{b}y^{2}}g(y)\right|=\left(\left|\mathcal{T}_{x}^{\mu, M}f(y)\right|^{p}|g(y)|^{q}\right)^{1/r}\left(\left|\mathcal{T}_{x}^{\mu, M}\right|^{p}\right)^{1/p-1/r}
	\left(|g(y)|^{q}\right)^{1/q-1/r},
	$$
	we have
	\begin{eqnarray*}
		\int_{\mathbb{R}}\left|\mathcal{T}_{x}^{\mu, M}f(y)e^{-i\frac{d}{b}y^{2}}g(y)\right||y|^{2\mu +1}dy&\leq&\left(\int_{\mathbb{R}}\left|\mathcal{T}_{x}^{\mu, M}f(y)\right|^{p}|g(y)|^{q}|y|^{2\mu +1}dy\right)^{\frac{1}{r}}\\
		&\times&\left(\int_{\mathbb{R}}\left|\mathcal{T}_{x}^{\mu, M}f(y)\right|^{p}|y|^{2\mu +1}dy\right)^{\frac{r-p}{rp}}\left(\int_{\mathbb{R}}|g(y)|^{q}|y|^{2\mu +1}dy\right)^{\frac{r-q}{rq}}.
	\end{eqnarray*}
	Which leads to
	\begin{eqnarray*}
		\left|f \star g(x)\right|^{r}&\leq&\left\|\mathcal{T}_{x}^{\mu, M}f\right\|_{\mu ,p} ^{r-p} \left\|g\right\|_{\mu , q} ^{r-q} \int_{\mathbb{R}}\left|\mathcal{T}_{x}^{\mu, M}f(y)\right|^{p}|g(y)|^{q}|y|^{2\mu +1}dy\\
	\end{eqnarray*}
	and according to the property $5$ of  Lemma \refeq{propertiestransl} we get 
	\begin{eqnarray*}
		\left|f \star g(x)\right|^{r}&\leq& \left\|f\right\|_{\mu ,p} ^{r-p} \left\|g\right\|_{\mu ,q} ^{r-q} \int_{\mathbb{R}}\left|\mathcal{T}_{x}^{\mu, M}f(y)\right|^{p}|g(y)|^{q}|y|^{2\mu +1}dy.
	\end{eqnarray*}
	Multiplying the above inequality by $x^{2\mu+1}$ and integrating on $\mathbb{R},$ we obtain
	\begin{eqnarray*}
		\left\|f \star g\right\|_{\mu ,r}^{r}&\leq&  \left\|f\right\|_{\mu ,p} ^{r-p} \left\|g\right\|_{\mu ,q} ^{r-q} \int_{\mathbb{R}} \left[\int_{\mathbb{R}}\left|\mathcal{T}_{x}^{\mu, M}f(y)\right|^{p}|g(y)|^{q}|y|^{2\mu +1}dy\right]|x|^{2\mu+1} dx\\
		&=&\left\|f\right\|_{\mu ,p} ^{r-p} \left\|g\right\|_{\mu ,q} ^{r-q} \int_{\mathbb{R}} |g(y)|^{q}|y|^{2\mu +1}\left[\int_{\mathbb{R}}\left|\mathcal{T}_{x}^{\mu, M}f(x)\right|^{p}|x|^{2\mu+1} dx\right]\\
		&\leq&\left\|f\right\|_{\mu ,p} ^{r} \left\|g\right\|_{\mu ,q} ^{r}.
	\end{eqnarray*}
\end{proof}

\section{Linear canonical Dunkl wavelet transform}\label{sec4}
In this section, we will introduce  the continuous   linear canonical Dunkl wavelet transform and we develop  its fundamental properties. 
\begin{defn}	
	A mother wavelet $\psi$ in $L^{2}_{\mu}(\mathbb{R})$ associated to the linear canonical Dunkl wavelet transform is
	said to be admissible if and only if it satisfies:
	\begin{equation}
		0<C_{\psi}^{M}:=|b|^{2\mu+2}\int_{0}^{\infty}\left|\mathcal{D}_{\mu}^{M}(f)\left(e^{-\frac{ia}{2b}z^{2}}\psi\right)(\lambda\xi)\right|^{2}\frac{d\xi}{\xi}<\infty\label{6.1}%
	\end{equation}
	where $C_{\psi}^{M}$ is called the admissibility condition of the linear canonical Dunkl wavelet transform.
\end{defn}
\begin{defn}
Let us define the family of the linear canonical Dunkl mother wavelets as follows:
\begin{equation}
	\forall y\in\mathbb{R},\quad\psi_{t,x}^{M}\left(  y\right)  =e^{-\frac
		{id}{2b}(x^{2}+y^{2})}e^{-\frac{ia}{2b}(y^{2}-x^{2})}\mathcal{T}_{x}^{\mu, M}(e^{\frac{id}{2b}z^{2}}\delta_{t}\psi)\left(  y\right)  .
	\label{6.2}%
\end{equation}
where $\delta_{t}$ is the dilation operator  defined for every measurable function $\psi$ on $\mathbb{R}$ and for all $t>0$ by 
\begin{equation}
\forall x\in\mathbb{R},\quad\delta_{t}\psi\left(  x\right)  =\frac{1}%
	{t^{\mu+1}}\psi\left(  \frac{x}{t}\right)  . \label{2.3}%
\end{equation}
\end{defn}
\begin{rem}
The proposed family of the linear canonical Dunkl mother wavelets can be rewrite in terms of the Dunkl translation operator and  the dilation operator $\delta_{t}$ as follows
\begin{equation}
	\forall y\in\mathbb{R},\quad\psi_{t,x}^{M}\left(  y\right)  =e^{-\frac
		{ia}{2b}(y^{2}-x^{2})}\mathcal{T}_{x}^{\mu}(\delta_{t}\psi)\left(  y\right)
	=e^{-\frac{ia}{2b}(y^{2}-x^{2})}\psi_{t,x}\left(  y\right)  \label{6.3}%
\end{equation}
where
\[
\forall y\in\mathbb{R},\quad\psi_{t,x}\left(  y\right)  =\mathcal{T}_{x}^{\mu}\delta_{t}
\psi\left(  y\right),
\]
and 
	\begin{equation}\label{Dunkltransl}
	\mathcal{T}_{x}^{\mu}\left(f\right)(y)=\int_{\mathbb{R}} f(z) f(z) \mathcal{W}_{\mu}(x, y, z) |z|^{2\mu +1}dz.
\end{equation}
\end{rem}

\begin{lem}
	Let $\psi\in L^{2}_{\mu}(\mathbb{R})$ be a linear canonical Dunkl mother wavelet. Then for all $\lambda\in\mathbb{R},$ we have
	\begin{equation}
		\mathcal{D}_{\mu}^{M}\left(  \psi_{t,x}^{M}\right)
		\left(  \lambda\right)  =\left\vert t\right\vert ^{\mu+1}e^{\frac{id}%
			{2b}\lambda^{2}}e^{\frac{ia}{2b}x^{2}}e^{-\frac{id}{2b}(t\lambda)^{2}}\mathcal{R}_{\mu}(i\lambda, x)\mathcal{D}_{\mu}^{M}\left[
		e^{-\frac{ia}{2b}z^{2}}\psi\right]  \left(  t\lambda\right)  . \label{6.4}%
	\end{equation}
	
\end{lem}

\begin{proof}
	Let $\psi\in L^{2}_{\mu}(\mathbb{R})$ be a linear canonical Dunkl mother wavelet, then we have for all $\lambda\in\mathbb{R}$
	\begin{equation*}
		\mathcal{D}_{\mu}^{M}\left(  \psi_{t,x}^{M}\right)
		\left(  \lambda\right)  =\frac{e^{\frac{id}{2b}\lambda^{2}}}%
		{(ib)^{\mu+1}}\mathcal{D}_{\mu}\left[  e^{\frac{ia}{2b}z^{2}}\psi
		_{t,x}^{M}\left(  z\right)  \right]  \left(  \frac{\lambda}%
		{b}\right).
	\end{equation*}
Hence, according to the relation  (\ref{6.3}), we get
	\begin{equation*}
	\mathcal{D}_{\mu}^{M}\left( \psi_{t,x}^{M}\right)
	\left(  \lambda\right) =\frac{e^{\frac{id}{2b}\lambda^{2}}}{(ib)^{\mu+1}}e^{\frac{ia}{2b}x^{2}%
	}\mathcal{D}_{\mu}\left( \psi_{t,x}\right)  \left(  \frac{\lambda}
	{b}\right)  .
\end{equation*}
On the other hand, we have
\begin{equation}\label{dunkoftransl}
	\forall \xi \in\mathbb{R}, \quad 	\mathcal{D}_{\mu}\left(\mathcal{T}_{x}^{\mu} f\right)(\xi)=\mathcal{R}_{\mu}(i\xi, x)\mathcal{D}_{\mu}(f)(\xi),
\end{equation}
and 
\begin{equation}\label{deltat1t}
	\forall y \in\mathbb{R}, \quad 	\mathcal{D}_{\mu}\left(\delta_t\psi\right)(y)=\delta_\frac{1}{t}(\psi)(y).
\end{equation}
Thus, using the relations (\ref{dunkoftransl}) and (\ref{deltat1t}), we obtain
	\begin{align*}
		\mathcal{D}_{\mu}\left(  \psi_{t,x}\right)  \left(\xi\right)   &
		=\mathcal{D}_{\mu}\left[  \mathcal{T}_{x}^{\mu}\delta_{t}\psi\right]  \left(
		\xi\right) \\
		&  =\mathcal{R}_{\mu}(i\xi, x)\mathcal{D}_{\mu}\left(  \delta_{t}%
		\psi\right)  \left(  \xi\right) \\
		&  =\mathcal{R}_{\mu}(i\xi, x)\delta_{\frac{1}{t}}\mathcal{D}_{\mu}\left(
		\psi\right)  \left(  \xi\right) \\
		&  =\left\vert t\right\vert ^{\mu+1}\mathcal{R}_{\mu}(i\xi, x)\mathcal{D}_{\mu}\left(  \psi\right)  \left(  t\xi\right)  .
	\end{align*}
	Therefore, we get
	\begin{align*}
		\mathcal{D}_{\mu}^{M}\left(  \psi_{t,x}^{M}\right)
		\left(  \lambda\right)   &  =\frac{e^{\frac{id}{2b}\lambda^{2}}}%
		{(ib)^{\mu+1}}e^{\frac{ia}{2b}x^{2}}\left\vert t\right\vert ^{\mu
			+1}\mathcal{R}_{\mu}(i\lambda, x)\mathcal{D}_{\mu}\left(  \psi\right)
		\left(  \frac{t\lambda}{b}\right) \\
		&  =\left\vert t\right\vert ^{\mu+1}e^{\frac{id}{2b}\lambda^{2}}%
		e^{\frac{ia}{2b}x^{2}}e^{-\frac{id}{2b}(t\lambda)^{2}}\mathcal{R}_{\mu}(i\lambda, x)\mathcal{D}_{\mu}^{M}\left[  e^{-\frac{ia}{2b}z^{2}}%
		\psi\right]  \left(  t\lambda\right)  .
	\end{align*}
	Which achieves the proof.
\end{proof}

\begin{defn} We define the  continuous linear canonical Dunkl wavelet transform of a function  $f$ in $ L^{2}_{\mu}(\mathbb{R})$ with respect to a
	mother wavelet $\psi$ in $L^{2}_{\mu}(\mathbb{R})$, as follows
	\begin{equation}
		\forall t>0,\forall x\in\mathbb{R},\quad \Phi_{\psi}^{D,M}f\left(
		t,x\right)  =\int_{\mathbb{R}}f\left(  z\right)  \overline{\psi
			_{t,x}^{M}\left(z\right)}|z|^{2\mu+1}dz.
		\label{6.5}%
	\end{equation}
\end{defn}
\begin{lem}
	Let  $f$ be a function  in $ L^{2}_{\mu}(\mathbb{R})$ and $\psi$ a mother wavelet  in $L^{2}_{\mu}(\mathbb{R})$. Then, we have
\begin{enumerate}
	\item $\forall t>0,\forall x\in\mathbb{R},\quad\Phi_{\psi}^{D,M}f\left(
	t,x\right) =e^{-\frac{ia}{2b}x^{2}}\left(  e^{\frac{ia}{2b}z^{2}}f\right)
	\ast\delta_{t}\overline{\psi}\left(  x\right). $
	\item  $ \forall t>0,\forall x\in\mathbb{R},\quad\Phi_{\psi}^{D,M}f\left(
	t,x\right) =e^{-\frac{ia}{2b}x^{2}}\left[  \delta_{\frac{1}{t}}\left(
	e^{\frac{ia}{2b}z^{2}}f\right)  \right]  \ast\overline{\psi}\left(
	\frac{x}{t}\right).$\\
	Here, $\ast$ designed the convolution product in the classical Dunkl setting.
\end{enumerate}
\end{lem}

\begin{proof}
	Let  $f$ be a function  in $ L^{2}_{\mu}(\mathbb{R})$ and $\psi$ a mother wavelet  in $L^{2}_{\mu}(\mathbb{R})$.
\begin{enumerate}
	\item For every $t>0$ and $x\in\mathbb{R}$, we have
	\begin{align*}
		\Phi_{\psi}^{D,M}f\left(t,x\right) &  =\int_{\mathbb{R}}f\left(  z\right)  e^{\frac{ia}{2b}(z^{2}-x^{2})}\overline{\psi_{t,x}\left(
			z\right)  }|z|^{2\mu+1}dz\\
		&  =e^{-\frac{ia}{2b}x^{2}}\int_{\mathbb{R}}e^{\frac{ia}{2b}z^{2}}f\left(
		z\right) \mathcal{T}_{x}^{\mu}\delta_{t}\overline{\psi}\left(  z\right)  |z|^{2\mu+1}dz \\
		&  =e^{-\frac{ia}{2b}x^{2}}\left(  e^{\frac{ia}{2b}z^{2}}f\right)
		\ast\delta_{t}\overline{\psi}\left(  x\right)  .
	\end{align*}
 \item Since, we have for all $t>0$ and $x\in\mathbb{R}$, 
 \begin{equation}\label{delta1/t} 
 	\delta_\frac{1}{t}\mathcal{T}_{x}^{\mu}= \mathcal{T}_\frac{{x}}{x}^{t}\delta_\frac{1}{t},
 \end{equation}
 thus, we get
 \begin{align*}
 		\Phi_{\psi}^{D,M}f\left(t,x\right)   &  =e^{-\frac{ia}%
 		{2b}x^{2}}\int_{\mathbb{R}}e^{\frac{ia}{2b}z^{2}}f\left(  z\right)  \mathcal{T}_{x}^{\mu}\delta_{t}\overline{\psi}\left(  z\right)  |z|^{2\mu+1}dz
 	\\
 	&  =e^{-\frac{ia}{2b}x^{2}}\int_{\mathbb{R}}\mathcal{T}_{x}^{\mu}\left(  e^{\frac{ia}%
 		{2b}z^{2}}f\right)  \left(  z\right)  \delta_{t}\overline{\psi}\left(
 	z\right)  |z|^{2\mu+1}dz \\
 	&  =e^{-\frac{ia}{2b}x^{2}}\int_{\mathbb{R}}\delta_{\frac{1}{t}}\mathcal{T}_{x}^{\mu}\left(  e^{\frac{ia}{2b}z^{2}}f\right)  \left(  z\right)  \overline{\psi
 	}\left(  z\right)  |z|^{2\mu+1}dz \\
 	&  =e^{-\frac{ia}{2b}x^{2}}\int_{\mathbb{R}}\mathcal{T}_{\frac{x}{t}}^{\mu}
 	\delta_{\frac{1}{t}}\left(  e^{\frac{ia}{2b}z^{2}}f\right)  \left(
 	z\right)  \overline{\psi}\left(  z\right)  |z|^{2\mu+1}dz \\
 	&  =e^{-\frac{ia}{2b}x^{2}}\delta_{\frac{1}{t}}\left(  e^{\frac{ia}%
 		{2b}z^{2}}f\right)  \ast\overline{\psi}\left(  \frac{x}{t}\right).
 \end{align*}
\end{enumerate}	
\end{proof}
\begin{rem}
	\begin{enumerate}
		\item  The continuous linear canonical Dunkl wavelet transform can also be written in terms of  continuous  Dunkl wavelet transform as follows
		\begin{equation}
			\forall t>0,\forall x\in\mathbb{R},\quad	\Phi_{\psi}^{D,M}f\left(t,x\right) =e^{-\frac{ia}{2b}x^{2}}	\Phi_{\psi}^{D}f\left(t,x\right)  \label{6.8}%
		\end{equation}
		where $\Phi_{\psi}^{D}f\left(t,x\right)$ is the continuous  Dunkl wavelet transform given by
		\begin{equation}
			\Phi_{\psi}^{D}f\left(t,x\right) =\int_{\mathbb{R}}f_{a,b}\left(  \xi\right)
			\overline{\psi_{t,x}\left(  \xi\right)  } |\xi|^{2\mu+1}d\xi
			=f_{a,b}\ast \delta_{t}\overline{\psi}\left(  \xi\right)  \label{6.9}%
		\end{equation}
		and%
		\begin{equation}
			\forall \xi\in\mathbb{R},\ f_{a,b}\left(  \xi\right)  =e^{\frac{ia}{2b}\xi^{2}%
			}f\left(  \xi\right). \label{6.10}%
		\end{equation}
	\item Let $\psi$ be a mother wavelet  in $L^{2}_{\mu}(\mathbb{R})$, then we have equality between the following norms
	\begin{equation}
		\forall t>0,\forall x\in\mathbb{R},\quad\left\Vert \psi_{t,x}^{M}\right\Vert
		_{\mu ,2}=\left\Vert \psi_{t,x}\right\Vert _{\mu ,2}=\left\Vert 	\mathcal{T}_{x}^{\mu}\delta_{t}\psi\right\Vert _{\mu ,2}\leq\left\Vert \psi\right\Vert
		_{\mu ,2}. \label{6.10a}%
	\end{equation}
	\end{enumerate}
\end{rem}

\begin{prop} Let  $f$ be a  function in $ L^{2}_{\mu}(\mathbb{R})$ and $\psi$ a	mother wavelet  in $L^{2}_{\mu}(\mathbb{R})$. Then we have for all 	$t>0$ and $x\in\mathbb{R}$
	\[\Phi_{\psi}^{D,M}f\left(t,x\right)   =\left\vert t\right\vert
	^{\mu+1}e^{-\frac{ia}{2b}x^{2}}\int_{\mathbb{R}}e^{-\frac{id}{2b}%
		\lambda^{2}}e^{\frac{id}{2b}(t\lambda)^{2}}  \mathcal{R}_{\mu}(i\lambda/b, x) \mathcal{D}_{\mu}^{M}f  \mathcal{D}_{\mu}^{M}\left[  e^{\frac{ia}{2b}z^{2}}\overline{\psi}\right]  \left(t\lambda\right) |\lambda|^{2\mu+1}d\lambda.\]
\end{prop}

\begin{proof}
From the definition of the linear canonical Dunkl transform (\ref{6.5}) and inversion formula (\ref{invlcdt}), we have for all $t>0$ and $x\in\mathbb{R}$ 
	\[\Phi_{\psi}^{D,M}f\left(t,x\right)   =\langle f,\ \psi
	_{t,x}^{M} \rangle_\mu =\langle\mathcal{D}_{\mu}^{M}f  ,\ \mathcal{D}_{\mu}^{M}\left(
	\psi_{t,x}^{M}\right)\rangle_\mu.
	\]
	According to the relation (\ref{6.4}), we obtain
	\begin{align*}
	\Phi_{\psi}^{D,M}f\left(t,x\right)   &  =\int_{\mathbb{R}}\mathcal{D}_{\mu}^{M}f  \left(  \lambda\right)
		\overline{\mathcal{D}_{\mu}^{M}\left(  \psi_{t,x}^{M}\right)  }\left(  \lambda\right) |\lambda|^{2\mu+1}d\lambda \\
		& = \left\vert t\right\vert
		^{\mu+1}e^{-\frac{ia}{2b}x^{2}}\int_{\mathbb{R}}e^{-\frac{id}{2b}%
			\lambda^{2}}e^{\frac{id}{2b}(t\lambda)^{2}}  \mathcal{R}_{\mu}(i\lambda/b, x) \mathcal{D}_{\mu}^{M}f  \mathcal{D}_{\mu}^{M}\left[  e^{\frac{ia}{2b}z^{2}}\overline{\psi}\right]  \left(t\lambda\right) |\lambda|^{2\mu+1}d\lambda.
	\end{align*}
\end{proof}
\begin{prop} Let  $f$ be a  function in $ L^{2}_{\mu}(\mathbb{R})$ and $\psi$ a	mother wavelet  in $L^{2}_{\mu}(\mathbb{R})$.
	Then, for all $t>0$ and for all $\lambda\in\mathbb{R},$ we have
	\[
	\mathcal{D}_{\mu}^{M}\left[  	\Phi_{\psi}^{D,M}f\left(  t,.\right)  \right]  \left(  \lambda\right)  =(-itb)^{\mu
		+1}e^{\frac{id}{2b}\left(  t\lambda\right)  ^{2}} \mathcal{D}_{\mu}^{M}f  \left(  \lambda\right)  \overline{ \mathcal{D}_{\mu}^{M}\left(  e^{-\frac{ia}{2b}z^{2}}\psi\right)  \left(
		t\lambda\right)  }.
	\]
\end{prop}

\begin{proof}  Let   $t>0$, $f$ be a  function in $ L^{2}_{\mu}(\mathbb{R})$ and $\psi$ a	mother wavelet  in $L^{2}_{\mu}(\mathbb{R})$.
	For all $\lambda\in\mathbb{R},$ we have
	\begin{align*}
		 \mathcal{D}_{\mu}^{M}\left[ 	\Phi_{\psi}^{D,M}f\left(  t,.\right)  \right]  \left(  \lambda\right) & =\mathcal{D}_{\mu}^{M}\left[  e^{-\frac{ia}{2b}x^{2}}\left(
		e^{\frac{ia}{2b}z^{2}}f\right)  \ast\delta_{t}\overline{\psi
		}\left(  x\right)  \right]  \left(  \lambda\right) \\
		&  =\frac{e^{\frac{id}{2b}\lambda^{2}}}{(ib)^{\mu+1}}\mathcal{D}_{\mu}\left[  \left(  e^{\frac{ia}{2b}z^{2}}f\right)  \ast\delta_{t}%
		\overline{\psi}\left(  x\right)  \right]  \left(  \frac{\lambda}{b}\right) \\
		&  =\frac{e^{\frac{id}{2b}\lambda^{2}}}{(ib)^{\mu+1}}\mathcal{D}_{\mu}\left(  e^{\frac{ia}{2b}z^{2}}f\right)  \left(  \frac{\lambda}{b}\right)
		\mathcal{D}_{\mu}\left(  \delta_{t}\overline{\psi}\right)  \left(
		\frac{\lambda}{b}\right) \\
		&  =\mathcal{D}_{\mu}^{M}f  \left(
		\lambda\right)  \mathcal{D}_{\mu}\left(  \delta_{t}\overline{\psi
		}\right)  \left(  \frac{\lambda}{b}\right) \\
		&  =\mathcal{D}_{\mu}^{M}f  \left(
		\lambda\right)  \delta_{\frac{1}{t}}\mathcal{D}_{\mu}\left(
		\overline{\psi}\right)  \left(  \frac{\lambda}{b}\right) \\
		&  =t^{\mu+1}\mathcal{D}_{\mu}^{M}f  \left(
		\lambda\right)  \overline{\mathcal{D}_{\mu}\left(  \psi\right)  \left(
			\frac{t\lambda}{b}\right)  }\\
		&  =(-itb)^{\mu+1}e^{\frac{id}{2b}\left(  t\lambda\right)  ^{2}%
		}\mathcal{D}_{\mu}^{M}f  \left(  \lambda\right)
		\overline{\mathcal{D}_{\mu}^{M}\left(  e^{-\frac{ia}{2b}z^{2}%
			}\psi\right)  \left(  t\lambda\right)  }.
	\end{align*}
	\end{proof}
In the following theorem, we derive the orthogonality relation corresponding to the continuous linear canonical Dunkl wavelet transform.
\begin{thm} 
		Let $\psi$ a mother wavelet  in $L^{2}_{\mu}(\mathbb{R})$  related to the the continuous linear canonical Dunkl wavelet transform. Then for all 
		$f,g\in L^{1}_{\mu}(\mathbb{R}) \cap L^{2}_{\mu}(\mathbb{R})$, we have
	\begin{equation}
		\int_{0}^{\infty}\int_{\mathbb{R}}\Phi_{\psi}^{D,M}f\left(
		t,x\right)  \overline{\Phi_{\psi}^{D,M}g\left(  t,x\right)
		} |x|^{2\mu+1} dx \frac{dt}{t}=C_{\psi}^{M}\int
		_{\mathbb{R}}f\left(  x\right)  \overline{g\left(  x\right)  }|x|^{2\mu+1} dx. \label{6.11}%
	\end{equation}
\end{thm}
\begin{proof}
	Let $\psi$ a mother wavelet  in $L^{2}_{\mu}(\mathbb{R})$  and	$f,g\in L^{1}_{\mu}(\mathbb{R}) \cap L^{2}_{\mu}(\mathbb{R})$. We have%
	\begin{align*}
		& \!\!\!\!\!\!\!\!\!\!\!\!\!\!\!\!\!\!\!\!\!\!\!\!\! \int_{\mathbb{R}}\Phi_{\psi}^{D,M}f\left(  t,x\right)
		\overline{\Phi_{\psi}^{D,M}g\left(  t,x\right)  }|x|^{2\mu+1} dx\\
		&  =\int_{\mathbb{R}}\mathcal{D}_{\mu}^{M}\left[
		\Phi_{\psi}^{D,M}f\left(  t,.\right)  \right]  \left(
		\lambda\right)  \overline{\mathcal{D}_{\mu}^{M}\left[
			\Phi_{\psi}^{D,M}g\left(  t,.\right)  \right]  \left(
			\lambda\right)  }|\lambda|^{2\mu+1} d\lambda \\
		&  =\int_{\mathbb{R}}(-itb)^{\mu+1}e^{\frac{id}{2b}\left(  t\lambda
			\right)  ^{2}}\mathcal{D}_{\mu}^{M}f  \left(
		\lambda\right)  \overline{\mathcal{D}_{\mu}^{M}\left(
			e^{-\frac{ia}{2b}z^{2}}\psi\right)  \left(  t\lambda\right)  }\\
		&  \times(itb)^{\mu+1}e^{-\frac{id}{2b}\left(  t\lambda\right)  ^{2}%
		}\overline{\mathcal{D}_{\mu}^{M}g  \left(
			\lambda\right)  }\mathcal{D}_{\mu}^{M}\left(  e^{-\frac{ia}%
			{2b}z^{2}}\psi\right)  \left(  t\lambda\right) |\lambda|^{2\mu+1} d\lambda \\
		&  =\int_{\mathbb{R}}\mathcal{D}_{\mu}^{M}f
		\left(  \lambda\right)  \overline{\mathcal{D}_{\mu}^{M}g  \left(  \lambda\right)  }\left\vert \mathscr F_{B,\mu
		}^{M}\left(  e^{-\frac{ia}{2b}z^{2}}\psi\right)  \left(
		t\lambda\right)  \right\vert ^{2}|\lambda|^{2\mu+1} d\lambda.
	\end{align*}
	Then, we obtain
	\begin{align*}
		& \!\!\!\!\!\!\!\!\!\!\!\!\!\!\!\!\!\!\!\!\!\!\!\!\!\!\!\!\!\! \int_{0}^{\infty}\int_{\mathbb{R}}\Phi_{\psi}^{D,M}f\left(
		t,x\right)  \overline{\Phi_{\psi}^{D,M}g\left(  t,x\right)
		}|x|^{2\mu+1} d\lambda\frac{dt}{t}\\
		&  =\int_{\mathbb{R}}\mathcal{D}_{\mu}^{M}f
		\left(  \lambda\right)  \overline{\mathcal{D}_{\mu}^{M}g  \left(  \lambda\right)  }\left(  \int_{0}^{\infty}\left\vert
		\mathcal{D}_{\mu}^{M}\left(  e^{-\frac{ia}{2b}z^{2}}\psi\right)
		\left(  t\lambda\right)  \right\vert ^{2}\frac{dt}{t}\right)  |\lambda|^{2\mu+1} d\lambda\\
		&  =C_{\psi}^{M}\int_{0}^{\infty}\mathcal{D}_{\mu}^{M}f  \left(  \lambda\right)  \overline
		{\mathcal{D}_{\mu}^{M}g  \left(  \lambda\right)
		}|\lambda|^{2\mu+1} d\lambda\\
		&  =C_{\psi}^{M}\int_{0}^{\infty}f\left(  x\right)  \overline
		{g\left(  x\right) }|x|^{2\mu+1}.
	\end{align*}
\end{proof}

From the previous theorem, we deduce the Plancherel type formula related to the linear canonical Dunkl wavelet transform.
\begin{cor}
	Let $\psi$ a mother wavelet  in $L^{2}_{\mu}(\mathbb{R})$  related to the the continuous linear canonical Dunkl wavelet transform. Then for all 
$f,g\in L^{1}_{\mu}(\mathbb{R}) \cap L^{2}_{\mu}(\mathbb{R})$, we have
	\begin{equation}
		\int_{0}^{\infty}\int_{\mathbb{R}}\left|\Phi_{\psi}^{D,M}f\left(
		t,x\right) \right|^2|x|^{2\mu+1} dx \frac{dt}{t}=C_{\psi}^{M}\left\Vert f\right\Vert _{\mu,2}^{2}.
		\label{6.12}%
	\end{equation}
\end{cor}

The following theorem ensures that the input signal may be recovered from the corresponding linear canonical Dunkl wavelet transform.
\begin{thm}
	Let $\psi$ a mother wavelet  in $L^{2}_{\mu}(\mathbb{R})$  related to the the continuous linear canonical Dunkl wavelet transform and
$f\in L^{1}_{\mu}(\mathbb{R}) \cap L^{2}_{\mu}(\mathbb{R})$.  Then $f$ can be reconstructed by the following formula
\begin{equation}
	f\left(  y\right)  =\frac{1}{C_{\psi}^{M}}\int_{0}^{\infty}\int
	_{\mathbb{R}}\Phi_{\psi}^{D,M}f\left(  t,x\right)  \psi
	_{t,x}^{M}\left(  y\right) |x|^{2\mu+1} dx  \frac
	{dt}{t} \label{6.13}%
\end{equation}
weakly in  $L^{2}_{\mu}(\mathbb{R})$.
\end{thm}

\begin{proof} 	Let $\psi$ a mother wavelet  in $L^{2}_{\mu}(\mathbb{R})$  related to the the continuous linear canonical Dunkl wavelet transform.
 Then for all 	$f\in L^{1}_{\mu}(\mathbb{R}) \cap L^{2}_{\mu}(\mathbb{R})$, we have
\begin{align*}
	& \!\!\!\!\!\!\!\!\!\!\!\! \int_{\mathbb{R}}\left(  \int_{0}^{\infty}\int_{\mathbb{R}}\Phi_{\psi}^{D,M}f\left(  t,x\right)  \psi_{t,x}^{M%
	}\left(  y\right)  |x|^{2\mu+1} dx  \frac{dt}{t}\right)
	\overline{g\left(  y\right)  }|y|^{2\mu+1} dy \\
	&  =\int_{0}^{\infty}\left(  \int_{\mathbb{R}}\Phi_{\psi}^{D,M}f\left(  t,x\right)  \left(  \int_{\mathbb{R}}\overline
	{g\left(  y\right)  \psi_{t,x}^{M}\left(  y\right)  }|y|^{2\mu+1} dy \right)  |x|^{2\mu+1} dx \right)  \frac{dt}{t}
	\\
	&  =\int_{0}^{\infty}\int_{\mathbb{R}}\Phi_{\psi}^{D,M}f\left(  t,x\right)  \overline{\Phi_{\psi}^{D,M}g\left(
		t,x\right)} |x|^{2\mu+1} dx  \frac{dt}{t}\\
	&  =C_{\psi}^{M}\int_{\mathbb{R}}f\left(  x\right)  \overline
	{g\left(  x\right)  }|x|^{2\mu+1} dx.
\end{align*}
Which completes the proof.
\end{proof}

\begin{thm}$\left(\text{Reproducing kernel}\right)$
	Let $\psi$ a mother wavelet  in $L^{2}_{\mu}(\mathbb{R})$  related to the the continuous linear canonical Dunkl wavelet transform  and $ M ^{\prime}=\left(  a^{\prime},
	b^{\prime}, c^{\prime}, d^{\prime} \right)  \in SL\left( 2, \mathbb{R}\right)  .$ Then the function 
\begin{equation}\label{kernelR}
		\mathcal{R}_{\psi}\left(  t,x,t^{\prime},x^{\prime}\right)  =\frac{1}{C_{\psi
		}^{ M }}\int_{\mathbb{R}}\psi_{t,x}^{ M }\left(  y\right)
	\overline{\psi_{t^{\prime},x^{\prime}}^{ M ^{\prime}}\left(  y\right)
	}|y|^{2\mu+1}dy
\end{equation}
	is point-wise bounded and  satisfies the following reproduction kernel formula
	\begin{equation}\label{reprokernelR}
	 \Phi_{\psi}^{D,M^{\prime}}f\left(  t^{\prime},x^{\prime}\right)
	=\int_{0}^{\infty}\int_{\mathbb{R}} \Phi_{\psi}^{D,M}f\left(  t,x\right)
	\mathcal{R}_{\psi}\left(  t,x,t^{\prime},x^{\prime}\right) |x|^{2\mu+1}dx \frac{dt}{t}.
   \end{equation}
\end{thm}

\begin{proof}
		According to the Cauchy-Schwartz inequality, we obtain the point-wise boundedness  of the kernel $\mathcal{R}_{\psi}$ (\ref{kernelR})
	\begin{align*}
		\left\vert \mathcal{R}_{\psi}\left(  t,x,t^{\prime},x^{\prime}\right)
		\right\vert  &  \leq\frac{1}{C_{\psi}^{ M }}\int_{0}^{\infty
		}\left\vert \psi_{t,x}^{ M }\left(  y\right)  \overline{\psi
			_{t^{\prime},x^{\prime}}^{ M ^{\prime}}\left(  y\right)  }\right\vert
		|y|^{2\mu+1}dy\\
		&  \leq\frac{1}{C_{\psi}^{ M }}\left\Vert \psi_{t,x}^{ M %
		}\right\Vert _{2,\mu}\left\Vert \psi_{t^{\prime},x^{\prime}}^{ M %
			^{\prime}}\right\Vert _{2,\mu}\\
		&  \leq\frac{1}{C_{\psi}^{ M }}\left\Vert \psi\right\Vert _{2,\mu
		}^{2}.
	\end{align*}
 Now, using the inversion formula (\ref{6.13}), we get the reproduction kernel formula (\ref{reprokernelR})
	\begin{align*}
		& \!\!\!\!\!\!\!\!\!\!\!\!\!\!  \Phi_{\psi}^{D,M^{\prime}}f\left(  t^{\prime},x^{\prime
		}\right) \\
		&  =\int_{\mathbb{R}}f\left(  y\right)  \overline{\psi_{t^{\prime},x^{\prime
			}}^{ M ^{\prime}}\left(  y\right)  }|y|^{2\mu+1}dy \\
		&  =\int_{\mathbb{R}}\left[  \frac{1}{C_{\psi}^{ M }}\int
		_{0}^{\infty}\int_{\mathbb{R}} \Phi_{\psi}^{D,M}f\left(
		t,x\right)  \psi_{t,x}^{ M }\left(  y\right)  |x|^{2\mu+1}dx \frac{dt}{t}\right]  \overline{\psi_{t^{\prime},x^{\prime}%
			}^{ M ^{\prime}}\left(  y\right)  }|y|^{2\mu+1}dy \\
		&  =\frac{1}{C_{\psi}^{ M }}\int_{0}^{\infty}\left[  \int_{\mathbb{R}} \Phi_{\psi}^{D,M}f\left(  t,x\right)  \int_{\mathbb{R}}\psi_{t,x}^{ M }\left(  y\right)  \overline
		{\psi_{t^{\prime},x^{\prime}}^{ M ^{\prime}}\left(  y\right)
		}|y|^{2\mu+1}dy  \right] |x|^{2\mu+1}dx
		\frac{dt}{t}\\
		&  =\int_{0}^{\infty}\int_{\mathbb{R}}\Phi_{\psi}^{D,M}f\left(  t,x\right)  \mathcal{R}_{\psi}\left(  t,x,t^{\prime},x^{\prime
		}\right)  |x|^{2\mu+1}dx  \frac{dt}{t}.
	\end{align*}
	Which completes the proof.
\end{proof}

\section{Applications}\label{sec5}
This section is devoted to elaborate some uncertainty inequalities for the continuous linear  canonical Dunkl wavelet transform.
We begin by establishing the following   important lemma, which will be
necessary for us  in the formulation of the main results in this part.
\begin{lem}
	\label{lem1} Let $\psi \in L^{p}_{\mu}(\mathbb{R})$ such that $1\leq p<+\infty$ and $f\in L^{1}_{\mu}(\mathbb{R})$. Then, we have
	\begin{equation}
		\left\Vert \Phi_{\psi}^{D,M}f\left(  t,.\right)  \right\Vert
		_{\mu,p}\leq t^{\left(  \mu +1\right)  \left(  \frac{2}{p}-1\right)
		}\left\Vert \psi\right\Vert _{\mu,p}\left\Vert f\right\Vert _{\mu,1}.
		\label{7.1}%
	\end{equation}
	
\end{lem}

\begin{proof}
Let $\psi \in L^{p}_{\mu}(\mathbb{R})$ such that $1\leq p<+\infty$ and $f\in L^{1}_{\mu}(\mathbb{R})$.
	According to  relations (\ref{6.8}) and (\ref{6.9}), we get
	\begin{align*}
		\left\Vert \Phi_{\psi}^{D,M}f\left(  t,.\right)  \right\Vert
		_{\mu,p}  &  =\left\Vert e^{-\frac{ia}{2b}x^{2}}\Phi_{\psi}^{D}
		f_{a,b}\left(  t,.\right)  \right\Vert _{\mu,p}\\
		&  =\left\Vert e^{-\frac{ia}{2b}x^{2}}f_{a,b}\ast\delta_{t}
		\overline{\psi}\left(  x\right)  \right\Vert _{\mu,p}%
	\end{align*}
	where $f_{a,b}$ is given by the relation (\ref{6.10}).\\
	Therefore 
		\begin{equation}
		\left\Vert \Phi_{\psi}^{D,M}f\left(  t,.\right)  \right\Vert
		_{\mu,p}=\left\Vert f\ast_{\mu}\delta_{t}\overline{\psi}\left(
		x\right)  \right\Vert _{\mu,p}. \label{7.2}%
	\end{equation}
	According to the Young's inequality, we obtain
	\begin{equation}
		\left\Vert \Phi_{\psi}^{D,M}f\left(  t,.\right)  \right\Vert
		_{\mu,p}\leq\left\Vert f\right\Vert _{\mu,1}\left\Vert \delta_{t}%
		\psi\right\Vert _{\mu,p}. \label{7.3}%
	\end{equation}
	On the other hand, we have
	\begin{align*}
		\left\Vert \delta_{t}\psi\right\Vert _{\mu,p}^{p}  &  =\int_{\mathbb{R}}|\delta_{t}\psi(x)|^{p} |x|^{2\mu+1} dx\\
		&  =\int_{\mathbb{R}}
		\frac{1}{t^{p\left(  \mu+1\right)  }}|\psi(\frac{x}{t})|^{p}x^{2\mu
			+1}dx\\
		&  =\frac{1}{t^{\left(
				p-2\right)  \left(  \mu+1\right)  }}\int_{\mathbb{R}}|\psi(u)|^{p}%
		u^{2\mu+1}du\\
		&  =t^{\left(  2-p\right)  \left(  \mu+1\right)  }\Vert\psi\Vert_{\mu,p}^{p}.
	\end{align*}
	Thus, we get
	\begin{equation}
		\left\Vert \delta_{t}\psi\right\Vert _{\mu,p}=t^{\left(  \mu
			+1\right)  \left(  \frac{2}{p}-1\right)  }\left\Vert \psi\right\Vert
		_{\mu,p}. \label{7.4}%
	\end{equation}
 Finally, combining  the relations (\ref{7.3}) and (\ref{7.4}) we obtain the desired inequality.
\end{proof}

\begin{lem}
	\label{lem2} Let $p,q\in [1,\infty)$, such that $\frac{1}{p}+\frac{1}{q}=1$. If
	$\psi\in L^{q}_{\mu}(\mathbb{R})$ and
	$f\in L^{p}_{\mu}(\mathbb{R})$ then $\Phi_{\psi}^{D,M}f\left( t,.\right)  \in L^{\infty}_{\mu}(\mathbb{R})$ and we have
	\begin{equation}
		\left\Vert \Phi_{\psi}^{D,M}f\left(  t,.\right)  \right\Vert
		_{\mu,\infty}\leq t^{\left(  \mu+1\right)  \left(  \frac{2}{q}-1\right)
		}\left\Vert \psi\right\Vert _{\mu,q}\left\Vert f\right\Vert _{\mu,p}.
		\label{7.5}%
	\end{equation}
	
\end{lem}
\begin{proof}
Let $p,q\in [1,\infty)$, such that $\frac{1}{p}+\frac{1}{q}=1$, $\psi\in L^{q}_{\mu}(\mathbb{R})$ and $f\in L^{p}_{\mu}(\mathbb{R})$. According to relations (\ref{7.3})
	and (\ref{7.4}) and the Young's inequality, we get
	\begin{align*}
		\left\Vert \Phi_{\psi}^{D,M}f\left(  t,.\right)  \right\Vert
		_{\mu,\infty}  &  =\left\Vert f\ast_{\mu}\delta_{t}\overline{\psi
		}\left(  x\right)  \right\Vert _{\mu,\infty}\\
		&  \leq\left\Vert f\right\Vert _{\mu,p}\left\Vert \delta_{t}%
		\overline{\psi}\right\Vert _{\mu,q}\\
		&  = t^{\left(  \mu+1\right)  \left(  \frac{2}{q}-1\right)  }\left\Vert
		f\right\Vert _{\mu,p}\left\Vert \psi\right\Vert _{\mu,q}.
	\end{align*}
	\end{proof}

\begin{thm}
	Let $\psi,f\in L^{2}_{\mu}(\mathbb{R})$ such that
	$\left\Vert \psi\right\Vert _{\mu,2}=\left\Vert f\right\Vert_{\mu,2}=1$
	and $\Omega\subset\mathbb{R}_{+}\times\mathbb{R}$ be a measurable subset satisfying
	\[
|\Omega|:=\int\int_{\Omega} |x|^{2\mu +1}dx \frac{dt}%
	{t}<+\infty
	\]
	and
	\[
	\forall\varepsilon>0,\quad\ \int\int_{\Omega}\left\vert  \Phi_{\psi}^{D,M}f\left(  t,x\right)  \right\vert ^{2}|x|^{2\mu +1}dx\frac
	{dt}{t}\geq1-\varepsilon.
	\]
	Then, we have
	\[
	\forall\varepsilon>0,\quad |\Omega| \geq1-\varepsilon.
	\]
	
\end{thm}

\begin{proof}
	Let $\psi,f\in L^{2}_{\mu}(\mathbb{R})$ such that
$\left\Vert \psi\right\Vert _{\mu,2}=\left\Vert f\right\Vert_{\mu,2}=1$.  We apply the Cauchy-Schwartz inequality, for all $t>0$ and $x\in \mathbb{R}$, we
	obtain
	\begin{align*}
		\left\vert \Phi_{\psi}^{D,M}f\left(  t,x\right)  \right\vert  &
		\leq\int_{\mathbb{R}}\left\vert f\left(  z\right)  \right\vert \left\vert
		\overline{\psi_{t,x}\left(  z\right)  }\right\vert |z|^{2\mu +1}dz
		\\
		&  \leq\left(  \int_{\mathbb{R}}\left\vert f\left(  z\right)  \right\vert
		^{2}|z|^{2\mu +1}dz \right)  ^{\frac{1}{2}}\left(  \int
		_{\mathbb{R}}\left\vert \overline{\psi_{t,x}\left(  z\right)  }\right\vert
		^{2}|z|^{2\mu +1}dz\right)  ^{\frac{1}{2}}\\
		&  \leq\left\Vert f\right\Vert _{\mu,2}\left\Vert \overline{\psi_{t,x}%
		}\right\Vert _{\mu,2}.
	\end{align*}
	According to relation (\ref{6.10a}), we get
	\begin{equation}
		\forall t>0,\forall x\in\mathbb{R},\quad \left\vert \Phi_{\psi}^{D,M}f\left(  t,x\right)  \right\vert \leq\left\Vert f\right\Vert _{2,\alpha
		}\left\Vert \psi\right\Vert _{\mu,2}=1. \label{7.9}%
	\end{equation}
	Therefore
	\[
	1-\varepsilon\leq\int\int_{\Omega}\left\vert \Phi_{\psi}^{D,M}f\left(  t,x\right)  \right\vert ^{2}|x|^{2\mu +1}dx \frac
	{dt}{t}\leq\int\int_{\Omega} |x|^{2\mu +1}dx \frac{dt}{t}.
	\]
	Which completes the proof.
\end{proof}

\begin{thm}
	Let $\psi$ and $\varphi$ two wavelets such that $0<C_{\psi}^{M},C_{\varphi}^{M}<+\infty.$ Then for all $p\in\left[  1,\infty
	\right)  $ and $f,g\in  L^{2}_{\mu}(\mathbb{R})$,
	we have
	\begin{align}
		& \!\!\!\!\!\!\!\!\!\!\!\!\! \left(  \int_{0}^{\infty}\int_{\mathbb{R}}\left\vert \Phi_{\psi}^{D,M}f\left(  t,x\right)  \Phi_{\varphi}^{D,M}g\left(
		t,x\right)  \right\vert ^{p}|x|^{2\mu +1}dx  \frac{dt}{t}\right)
		^{\frac{1}{p}}\label{7.10}\\
		&  \leq\left(  C_{\psi}^{ M }C_{\varphi}^{ M }\right)
		^{\frac{1}{2p}}\left(  \left\Vert \psi\right\Vert _{\mu,2}\left\Vert
		\varphi\right\Vert _{\mu,2}\right)  ^{\frac{p-1}{p}}\left\Vert f\right\Vert
		_{\mu,2}\left\Vert g\right\Vert _{\mu,2}.\nonumber
	\end{align}
	
\end{thm}

\begin{proof}
According to  Cauchy-Schwartz inequality and the relation (\ref{6.12}), we obtain
	\begin{align*}
		&\!\!\!\!\!\!\!\!\!\!\!\!\!  \int_{0}^{\infty}\int_{\mathbb{R}}\left\vert \Phi_{\psi}^{D,M}f\left(  t,x\right)  \Phi_{\varphi}^{D,M}g\left(
		t,x\right)  \right\vert |x|^{2\mu +1}dx  \frac{dt}{t}\\
		&  \leq\left(  \int_{0}^{\infty}\int_{\mathbb{R}}\left\vert \Phi_{\psi}^{D,M}f\left(  t,x\right)  \right\vert ^{2}|x|^{2\mu +1}dx \frac{dt}{t}\right)  ^{\frac{1}{2}}\left(   \int_{0}^{\infty}\int_{\mathbb{R}}\left\vert \Phi_{\varphi}^{D,M}g\left(
		t,x\right)  \right\vert ^{2}|x|^{2\mu +1}dx  \frac{dt}{t}\right)
		^{\frac{1}{2}}\\
		&  \leq\sqrt{C_{\psi}^{M}}\sqrt{C_{\varphi}^{M}}\left\Vert
		f\right\Vert _{\mu,2}\left\Vert g\right\Vert _{\mu,2}.
	\end{align*}
	Therefore, according to the relation (\ref{7.9}), we get
	\begin{align*}
		& \!\!\!\!\!\! \left(  \int_{0}^{\infty}\int_{\mathbb{R}}\left\vert \Phi_{\psi}^{D,M}f\left(  t,x\right)  \Phi_{\varphi}^{D,M}g\left(
		t,x\right)  \right\vert ^{p} |x|^{2\mu +1}dx  \frac{dt}{t}\right)
		^{\frac{1}{p}}\\
		&  \leq\left(  \int_{0}^{\infty}\int_{\mathbb{R}}\left(  \left\Vert
		f\right\Vert _{\mu,2}\left\Vert \psi\right\Vert _{\mu,2}\left\Vert
		\varphi\right\Vert _{\mu,2}\left\Vert g\right\Vert _{\mu,2}\right)
		^{p-1}\left\vert \Phi_{\psi}^{D,M}f\left(  t,x\right)
		\Phi_{\varphi}^{D,M}g\left(  t,x\right)  \right\vert |x|^{2\mu +1}dx \frac{dt}{t}\right)  ^{\frac{1}{p}}\\
		&  \leq\left(  \left\Vert f\right\Vert _{\mu,2}\left\Vert \psi\right\Vert
		_{\mu,2}\left\Vert \varphi\right\Vert _{\mu,2}\left\Vert g\right\Vert
		_{\mu,2}\right)  ^{\frac{p-1}{p}}\left(  \int_{0}^{\infty}\int_{\mathbb{R}}\left\vert \Phi_{\psi}^{D,M}f\left(  t,x\right)
		\Phi_{\varphi}^{D,M}g\left(  t,x\right)  \right\vert|x|^{2\mu +1}dx \frac{dt}{t}\right)  ^{\frac{1}{p}}\\
		&  \leq\left(  \left\Vert f\right\Vert _{\mu,2}\left\Vert \psi\right\Vert
		_{\mu,2}\left\Vert \varphi\right\Vert _{\mu,2}\left\Vert g\right\Vert
		_{\mu,2}\right)  ^{\frac{p-1}{p}}\left(  \sqrt{C_{\psi}^{ M }}%
		\sqrt{C_{\varphi}^{ M }}\left\Vert f\right\Vert _{\mu,2}\left\Vert
		g\right\Vert _{\mu,2}\right)  ^{\frac{1}{p}}\\
		&  \leq\left(  C_{\psi}^{ M }C_{\varphi}^{ M }\right)
		^{\frac{1}{2p}}\left(  \left\Vert \psi\right\Vert _{\mu,2}\left\Vert
		\varphi\right\Vert _{\mu,2}\right)  ^{\frac{p-1}{p}}\left\Vert f\right\Vert
		_{\mu,2}\left\Vert g\right\Vert _{\mu,2}.
	\end{align*}
	
\end{proof}
Consequently, we have the following result.
\begin{cor}
	Let $\psi$ be a wavelet satisfying $0<C_{\psi}^{ M }<+\infty.$ Then
	for all $p\in\left[2,\infty\right) $ and $f\in  L^{2}_{\mu}(\mathbb{R}) $, we have
	\begin{equation}
		\left(  \int_{0}^{\infty}\int_{\mathbb{R}}\left\vert \Phi_{\psi}^{D,M}f\left(  t,x\right)  \right\vert ^{p}|x|^{2\mu +1}dx 
		 \frac{dt}{t}\right)  ^{\frac{1}{p}}\leq\left(  C_{\psi}%
		^{ M }\right)  ^{\frac{1}{p}}\left\Vert \psi\right\Vert _{\mu,2}^{\frac{p-2}{p}}\left\Vert f\right\Vert _{\mu,2}. \label{7.11}%
	\end{equation}
	\end{cor}

\begin{thm}
	Let $\psi,f\in  L^{2}_{\mu}(\mathbb{R})$ such that
	$\left\Vert \psi\right\Vert _{2,\alpha}=\left\Vert f\right\Vert _{2,\alpha}=1$
	and $\Omega\subset\mathbb{R}_{+}\times\mathbb{R}$ be a measurable subset such that
	\[
	|\Omega| <\infty\quad\text{ and }\quad\int\int_{\Omega}\left\vert\Phi_{\psi}^{D,M}f\left(  t,x\right)  \right\vert ^{2}|x|^{2\mu +1}dx 
	 \frac{dt}{t}\geq1-\varepsilon,\quad\ \forall\varepsilon\in\left(0,1\right).
	\]
	Then for all $p\in (2,\infty)$, we have
	\[
	|\Omega|\geq\left(  1-\varepsilon\right)
	^{\frac{p}{p-2}}\left(  C_{\psi}^{ M }\right)  ^{\frac{2}{2-p}},\quad  \forall\varepsilon\in\left(0,1\right).
	\]
	
\end{thm}

\begin{proof}
	Let  $p\in (2,\infty)$. Then,  according to the H\"{o}lder inequality, we obtain
	\begin{align*}
		 & \!\!\!\!\!\!\!\!\!\!\!\!\!\!\!\!\!\!\!\!  \int\int_{\Omega}\left\vert  \Phi_{\psi}^{D,M}f\left(  t,x\right)
		\right\vert ^{2}|x|^{2\mu +1}dx \frac{dt}{t}\\
		&  \leq\left(  \int\int_{\Omega}\left\vert  \Phi_{\psi}^{D,M}f\left(
		t,x\right)  \right\vert ^{p}|x|^{2\mu +1}dx  \frac{dt}{t}\right)
		^{\frac{2}{p}}\left(  \int\int_{\Omega}|x|^{2\mu +1}dx \frac{dt}%
		{t}\right)  ^{\frac{p-2}{p}}\\
		&  \leq |\Omega|^{\frac{p-2}{p}}\left(  \int
		_{0}^{\infty}\int_{\mathbb{R}}\left\vert \Phi_{\psi}^{D,M}f\left(  t,x\right)  \right\vert ^{p}|x|^{2\mu +1}dx  \frac
		{dt}{t}\right)  ^{\frac{2}{p}}.
	\end{align*}
	Now, from the relation (\ref{7.11}), we obtain
	\[
	1-\varepsilon\leq\int\int_{\Omega}\left\vert \Phi_{\psi}^{D,M}f\left(  t,x\right)  \right\vert ^{2}|x|^{2\mu +1}dx \frac
	{dt}{t}\leq |\Omega|^{\frac{p-2}{p}}\left(
	\left(  C_{\psi}^{ M }\right)  ^{\frac{1}{p}}\left\Vert \psi
	\right\Vert_{\mu,2}^{\frac{p-2}{p}}\left\Vert f\right\Vert_{\mu,2}\right)  ^{2}.
	\]
	Therefore
	\[
	\forall\varepsilon\in\left(0,1\right),\quad 0<1-\varepsilon\leq\left(
	C_{\psi}^{ M }\right)  ^{\frac{2}{p}}|\Omega|^{\frac{p-2}{p}}%
	\]
	and
	\[
\forall\varepsilon\in\left(0,1\right),\quad \left(  1-\varepsilon\right)
	^{\frac{p}{p-2}}\left(  C_{\psi}^{ M }\right)  ^{\frac{2}{2-p}}\leq
	|\Omega|.
	\]
\end{proof}

\section*{Summary}
We have established the generalized translation operator and the generalized convolution production in the linear canonical Dunkl domains which allowed us to construct a new type of generalized linear canonical wavelet transform. We have  investigated the fundamentals properties for this new generalized wavelet transform such as orthogonality relation, Plancherel type formula, reconstruction formula and reproducing kernel. In the end,  we have  investigated  some uncertainty inequalities for the desired wavelet transform as applications.

\section*{Conclusion}

In this paper, we presented a novel concept of translation operator and convolution product in the context of the linear canonical Dunkl transform which generate a wide class of integral transforms, which facilitate for further research works in the following extent. Also, we have accomplished the linear canonical Dunkl wavelet transform and have  investigated its fundamentals properties and some uncertainty inequalities, which facilitate for further research works for a general wide class fo wavelet transforms.
In our future works we shall study the localization operators for this new wavelet transform and and investigate  others uncertainty inequalities.

\section*{Declarations}

\textbf{Conflict of interest} The  author confirms that there is no conflict of interest.

\bibliographystyle{abbrv}
\bibliography{LCDWT}


\end{document}